\documentclass[10pt]{amsart}
\usepackage{amssymb}
\usepackage[headings]{fullpage}
\usepackage{amsfonts}
\usepackage{paralist}
\usepackage{color}

\usepackage[colorlinks,pagebackref=true,pdftex]{hyperref}

\makeatletter
\@addtoreset{equation}{section}
\def\theequation{\thesection.\@arabic \c@equation}

\def\theenumi{\@roman\c@enumi}
\def\@citecolor{blue}
\def\@linkcolor{blue}
\def\@urlcolor{blue}

\makeatother

\newtheorem{theorem}[equation]{Theorem}
\newtheorem{lemma}[equation]{Lemma}
\newtheorem{proposition}[equation]{Proposition}

\newtheorem{claim*}{Claim}

\theoremstyle{definition}
\newtheorem{remark}[equation]{Remark}

\newtheorem{eg}[equation]{Example}

\newtheorem{definition}[equation]{Definition}

\newtheorem{notn}[equation]{Notation}

\newcommand{\bbB}{{\mathbb B}}

\newcommand{\naturals}{\mathbb N}
\newcommand{\ints}{\mathbb Z}
\newcommand{\define}{\emph}
\DeclareMathOperator{\Spec}{Spec}
\DeclareMathOperator{\height}{ht}
\DeclareMathOperator{\codim}{codim}
\DeclareMathOperator{\projdim}{pd}
\DeclareMathOperator{\Tor}{Tor}

\newsavebox{\upperboundtheorem}

\title{Bounds for the Multiplicity of Gorenstein algebras}

\author[S.~El~Khoury]{Sabine El Khoury} 
\address{Department of Mathematics, American University of Beirut, Beirut,
Lebanon.}
\email{se24@aub.edu.lb} 

\author[M.~Kummini]{Manoj Kummini}
\address{Chennai Mathematical Institute, Siruseri,
Tamilnadu 603103 India.}
\email{mkummini@cmi.ac.in}

\author[H.~Srinivasan]{Hema Srinivasan}
\address{Department of Mathematics, University of Missouri, Columbia,
Missouri, USA.}
\email{hema@math.missouri.edu}

\thanks{
The first author thanks the American University of Beirut for supporting
part of this work through a long term development grant, and Mathematical
Sciences Research Institute, Berkeley CA for hospitality. The second author
thanks MSRI for support during Fall 2012. In addition, they thank 
the University of Missouri, Columbia, MO for hospitality.}

\begin{document}

\begin{abstract}
We prove upper bounds for the Hilbert-Samuel multiplicity of standard
graded Gorenstein algebras. The main tool that we use is Boij-S\"oderberg
theory to obtain a decomposition of the Betti table of a Gorenstein algebra
as the sum of rational multiples of symmetrized pure tables. Our bound
agrees with the one in the quasi-pure case obtained by Srinivasan 
[J. Algebra, vol.~208, no.~2, (1998)].
\end{abstract}

\maketitle

\section{Introduction}
Let $\Bbbk$ be a field and $R$ a standard graded polynomial ring over
$\Bbbk$, i.e., $R=\bigoplus_{i \in \naturals} R_i$ as
$\Bbbk$-vector-spaces, $R_1$ is a finite dimensional $\Bbbk$-vector-space
and $R$ is generated as a $\Bbbk$-algebra by $R_1$. Let $M$ be a finitely
generated graded $R$-module and $e(M)$ the 
Hilbert-Samuel multiplicity of $M$. We say that
$M$ is \define{Gorenstein} if $M$ is Cohen-Macaulay and a minimal
$R$-free resolution of $M$ is self-dual.
In this paper, we prove the following theorem:

\def\upperboundtheorem{Let $M$ be a graded Gorenstein $R$-module that is
minimally generated homogeneous elements of degree zero.
Let $s = \codim M$ and $k = \lfloor \frac{s}{2} \rfloor$.
Let $\beta_0 (M)$ denote the minimal number of generators of $M$.
 For $0 \leq i \leq s$, write 
$m_i = m_i(M) = \min\{j : \Tor_i^R(\Bbbk, M)_j \neq 0\}$ and
$M_i =M_i(M) = \max\{j : \Tor_i^R(\Bbbk, M)_j \neq 0\}$.
Then
\[
e(M) \leq \frac{\beta_0(M)}{s!}
\prod_{i=1}^k \min\left\{M_i,
{\left\lfloor\frac{m_s}{2}\right\rfloor}\right\}
\prod_{i=k+1}^s \max\left\{m_{i}, 
{\left\lceil\frac{m_s}{2}\right\rceil}\right\}.
\]}

\begin{theorem}
\label{theorem:UpperBound}
\upperboundtheorem
\end{theorem}

There is a history of looking for bounds for $e(M)$ in terms of the
homological invariants $m_i(M)$ and $M_i(M)$.  Note that these are,
respectively, the minimum and maximum twists in the free modules in a
minimal $R$-free resolution of $M$.  When $M = R/I$ for a homogeneous
ideal $I$ of $R$, then the conjectures of C.~Huneke and Srinivasan,
see~\cite[Conjecture~1]{HS}, and of J.~Herzog and Srinivasan,
see~\cite[Conjecture~2]{HS}, proposed bounds for $e(M)$. These conjectures
were proved using the more general framework of Boij-S\"oderberg theory. The theory refers to the study of the decomposition of Betti tables of $M$
generated graded $R$-modules in terms of extremal rays in the cone
generated by the Betti tables of all the finitely generated $R$-modules, in
a certain infinite dimensional rational vector space. This decomposition was
conjectured by M.~Boij and J.~S\"oderberg~\cite[Conjecture~2.4]{BS08}, who
showed that if this stronger conjecture were true, then the conjecture of
Huneke and Srinivasan~\cite[Conjecture~1]{HS} and more
would hold. This was in
turn proved by D.~Eisenbud, G.~Fl{\o}ystad and
J.~Weyman~\cite{EFWpureResln07} in characteristic zero and Eisenbud and
F.-O.~Schreyer~\cite{ES} in a characteristic-free situation. Thereafter,
Boij and {S\"oderberg} proved that the conjecture of Herzog and 
Srinivasan~\cite[Conjecture~2]{HS} holds.
See~\cite{ES:ICMsurvey} for a survey of Boij-S\"oderberg theory.
We will need to use the Boij-S\"oderberg decomposition in our arguments; we
have summarized the relevant features in Section~\ref{sec:prelims}.

When the resolution is quasi-pure, 
i.e.  $m_i \geq M_{i-1}$ for all $i=2, \cdots, s$,
the aforementioned conjecture of Huneke
and Srinivasan were proved using the equations of Peskine-Szpiro 
coming from the additivity of the Hilbert function \cite[Theorem~1.2]{HS}.  
The duality of the resolution in the Gorenstein case can be exploited to
obtain stronger bounds.
In~\cite[Theorem~4]{S}, Srinivasan showed that if $R/I$ is Gorenstein
and has a quasi-pure $R$-resolution, then
\begin{equation}
\label{equation:SrinivasanQuasiPure}
\frac{m_1 \cdots m_k M_{k+1} \cdots M_s}{s!} \leq
e(R/I) \leq
\frac{M_1 \cdots M_k m_{k+1} \cdots m_s}{s!}.
\end{equation}
The notation here is the same as in Theorem~\ref{theorem:UpperBound}: 
$s = \height I$ and $k = \lfloor \frac{s}{2} \rfloor$. 
Theorem~\ref{theorem:UpperBound}  generalizes the upper bound in this result
to Gorenstein algebras with arbitrary resolutions. 

Let $d = \dim M$. The \define {Hilbert coefficients}  $e_i(M)$ with $0 \leq
i \leq d$ are defined by expressing the Hilbert polynomial of $M$ in the
form
\[
t \mapsto \sum_{i=1}^d (-1)^ie_i(M) \binom{t+d-1-i}{d-1-i}.
\]
For each $i$, $e_i(M)$ is an integer and $e_0(M) = e(M)$. 
Using Boij-S\"oderberg theory, 
Herzog and X.~Zheng obtained  upper and lower bounds, in the fashion
of~\cite[Conjecture~1]{HS}, for all the Hilbert coefficients of arbitrary
Cohen-Macaulay graded $R$-modules~\cite[Theorem~2.1]{HZ}.
Extending~\cite[Theorem~4]{S}, El~Khoury and Srinivasan obtained upper and
lower bounds for all Hilbert coefficients of Gorenstein quotient rings
of $R$ that have quasi-pure $R$-resolutions~\cite[Theorem~4.2]{ES1}.

In this paper, we show that the duality in the minimal free resolution of a
Gorenstein algebra can be captured in the 
Boij-S\"oderberg decomposition of its Betti table; this can be used to
obtain stronger bounds for multiplicity. We define and use the notion of
symmetrized pure Betti tables to capture the duality in the decomposition.
Our upper bound recovers the upper bound
in~\eqref{equation:SrinivasanQuasiPure}. 

The computer algebra system
\texttt{Macaulay2}~\cite{M2} 
provided valuable assistance in studying examples.  

\section{Preliminaries}
\label{sec:prelims}

\subsection*{Notation}

As earlier, $\Bbbk$ is a field and $R = \Bbbk[x_1, \ldots, x_n]$ is a
$n$-dimensional polynomial ring over $\Bbbk$ with $\deg x_i = 1$ for all 
$1 \leq i \leq n$. Let $M$ be a finitely generated graded $R$-module.
The \define{codimension} of $M$, denoted $\codim M$, is the codimension of
the support of $M$ in $\Spec R$.
The \define{graded Betti numbers} of $M$ are 
$\beta_{i,j}(M) = \dim_\Bbbk \Tor_i^R(\Bbbk, M)_j$. Note that
$\beta_{i,j}(M)$ is the number of copies of $R(-j)$ that appear at
homological degree $i$, in a minimal $R$-free resolution of $M$.
We think of the collection
$\{\beta_{i,j}(M) : 0 \leq i \leq n, j \in \ints\}$ as an element
\[
\beta(M) = \left(\beta_{i,j}(M)\right)_{0 \leq i \leq n, j \in \naturals} 
\in \bbB := 
\bigoplus\limits_{i=0}^n \bigoplus\limits_{j \in  \mathbb{Z}} \mathbb{Q},
\]
and call it the \emph{Betti table} of $M$.  
In general, 
a \define{rational Betti table} $\beta$ is an element
$\beta = \left(\beta_{i,j}\right)_{0 \leq i \leq n, j \in \naturals} \in \bbB$ 
such that: 
\begin{asparaenum}
\item
for all $0 \leq i \leq n$, $\beta_{i,j} = 0$ for all $j$ such
that $|j| \gg 0$,
\item for all $i>0$ and for all $j$, if 
$\beta_{i,j} \neq 0$ then there exists $j' < j$ such that $\beta_{i-1,j'}
\neq 0$.
\end{asparaenum}
Let
$\beta = \left(\beta_{i,j}\right)_{0 \leq i \leq n, j \in \naturals}$ be a
rational Betti table.
Its \define{length} is $\max \{i : \beta_{i,j} \neq 0 \;\text{for some}\; j\}$.
For $0 \leq i \leq \mathrm{length}(\beta)$, write 
$t_i(\beta) = \min\{j :  \beta_{i,j} \neq 0\}$ and
$T_i(\beta) = \max\{j :  \beta_{i,j} \neq 0\}$.

\subsection*{Boij-S\"oderberg theory}

We give here a summary of Boij-S\"oderberg theory that is relevant to us.
For details, see~\cite{ES, ES:ICMsurvey}; an expository account
is~\cite{FloyBSSurv11}.

A \define{degree sequence} of \define{length} $s$ is an increasing sequence
$d = (d_0 < d_1 < \cdots < d_s)$ of integers. For every such degree
sequence $d$, there is a finitely generated graded Cohen-Macaulay
$R$-module $M$ of codimension $s$ such that for all $0 \leq i \leq s$,
$\beta_{i,j}(M) \neq 0$ if and only if $j = d_i$; for such a module $M$, we
will say that $d$ is the \define{type} of its resolution.  
Moreover, 
by the Herzog-K\"uhl equations~\cite{HK},
$\beta(M)$ is a positive rational multiple of the pure Betti
table, which we denote by $\beta(d)$, given by:
\begin{equation}
\label{equation:HerzogKuhl}
\beta(d)_{i,j} = 
\begin{cases}
\frac{1}{\prod_{l\neq i} |d_l-d_i|}, & 0 \leq i \leq s \;\text{and}\; j=d_i \\
0, & \;\text{otherwise}.
\end{cases}
\end{equation}
 
Let  $d = (d_0 < \cdots < d_s)$ be a degree sequence.  We call the Betti
table $\beta(d)$ defined in~\eqref{equation:HerzogKuhl}, the \define{pure
Betti table} associated to $d$. For $0 \leq i \leq s$, 
write $\beta_i(d) = \beta(d)_{i,d_i}$.
Eisenbud, {Fl{\o}ystad} and Weyman~\cite[Theorem~0.1]{EFWpureResln07} (in
characteristic zero) and Eisenbud and
Schreyer~\cite[Theorem~0.1]{ES} showed that for all degree sequences $d$,
there is a Cohen-Macaulay $R$-module $M$ such that $\beta(M)$ is a rational
multiple of $\beta(d)$. Moreover, for all $R$-modules $M$,
$\beta(M)$ can be written as a non-negative rational combination of the
$\beta(d)$; if we take a saturated chain of degree sequences, then the
non-negative rational coefficients in the decomposition are
unique~\cite[Theorem~0.2]{ES}.

\subsection*{Self-dual resolutions and symmetrized Betti tables}

Let $\beta$ be a Betti table.  Let $s$ and $N$ be integers.
We say that $\beta$ is 
\define{$(s,N)$-self-dual} if $\beta_{i,j} = \beta_{s-i, N-j}$ for all $i,j$. 
We say that $\beta$ is \define{self-dual} 
if there exist $s$ and $N$ such that $\beta$ is $(s,N)$-self-dual.
If $\beta$ is self-dual, then $s$ is the length of $\beta$ and
$N = \max\{j : \beta_{s,j} \neq 0\} + \min\{j : \beta_{0,j} \neq 0\}$.

\begin{definition} 
\label{definition:symmetrized}
Let  $d = (d_0 < \cdots < d_s)$ be a degree sequence and $N \geq d_0+d_s$.
Let $d^{\vee, N} = (N-d_s< \cdots < N-d_0)$. 
Denote the pure Betti table
associated to $d^{\vee, N}$ by $\beta^{\vee, N}(d)$. Similarly, set 
$\beta^{\vee, N}_i(d) = \beta^{\vee, N}(d)_{i, N-d_{s-i}}$. 
Let $\beta_{\mathrm{sym}}(d,N) = \beta(d) + \beta^{\vee, N}(d)$. 
We call $\beta_{\mathrm{sym}}(d,N) $ the \define{symmetrized pure Betti
table}, given by symmetrizing $d$ with respect to $N$.
\end{definition}

C.~Peskine and L.~Szpiro~\cite{PS} observed that for a finitely generated
graded $R$-module $M$ with $s=\codim M$,
\[
\sum_{i=0}^{\projdim M} \sum_{j} (-1)^i\beta_{i,j}j^l = 
\begin{cases}
0 & \;\text{if}\; 0 \leq l < s \\  
(-1)^ss!e(M) & \;\text{if}\; l = s.
\end{cases}
\]
(Here $\projdim M$ is the projective dimension of $M$, or equivalently, the
length of $\beta(M)$.) Suppose that $d$ is a degree sequence of length $s$.
Since the pure Betti table $\beta(d)$ is, up to
multiplication by a rational number, the Betti table of a Cohen-Macaulay
$R$-module of codimension $s$, we see  that 
$\sum_{i=0}^s (-1)^i\beta_{i}(d)d_i^l = 0$ for all $0 \leq l < s$.
Further, by direct calculation using~\eqref{equation:HerzogKuhl}
we can see that$\sum_{i=0}^s (-1)^i\beta_{i}(d)d_i^s = (-1)^s$. 
Therefore we set
\begin{equation}
\label{equation:multPureBettiTable}
e(\beta(d)) = \frac{1}{s!}  \quad\text{and}\quad
e (\beta_{\mathrm{sym}}(d,N)) = \frac{2}{s!}.
\end{equation}

We now argue that the Betti table of a Gorenstein module can be
decomposed into a non-negative rational combination of symmetrized pure
Betti tables.

\begin{proposition}
\label{proposition:moduleBettiDec}
Let $M$ be finitely generated graded Cohen-Macaulay $R$-module with $\codim
M = s$, generated minimally by homogeneous elements of degree zero. 
Suppose that $\beta(M)$ is self-dual. 
Let $N =  T_s(M)$.
Then there exist degree sequences $d^\alpha, 0 \leq \alpha \leq a$ for some
$a \in \naturals$ and positive rational numbers $r_\alpha, 0 \leq \alpha \leq a$ such
that 
\[
\beta(M) = \sum_{\alpha=0}^a r_\alpha\beta_{\mathrm{sym}}(d^\alpha, N).
\]
Moreover,
\begin{enumerate}  
\item the $d^\alpha$ are degree sequences of length $s$ and they are not
$(s,N)$-dual to each other.
\item $d^{\alpha+1} > d^{\alpha}$ for all $0 \leq \alpha \leq a-1$.
\item $N \ge d^{\alpha}_i+d^{\alpha}_{s-i} $ for all $\alpha$ and $i$, or
equivalently, $d^{\alpha} \leq (d^{\alpha})^{\vee, N}$ for all $\alpha$.
\end{enumerate}
\end{proposition}

\begin{proof}
We prove this similar to the Decomposition Algorithm of~\cite
[p.~864]{ES}. First, for the duration of this proof, we will say that
a rational Betti table $\beta \geq 0$ if $\beta_{i,j} \geq 0$ for all
$i,j$.  
Let $d^0 = (0 < t_1(M) < \cdots < t_s(M)=N)$. Then there
exist positive rational numbers $r_0$ and $r'_0$ such that 
$\beta(M) - r_0\beta(d^0) + r'_0\beta((d^0)^{\vee, N}) \geq 0$.

Let $d^{\alpha}, 0
\leq \alpha \leq a$ be a maximal (by inclusion) saturated chain of degree
sequences with $d^{\alpha}_0 = 0$ and $d^{\alpha}_s = N$ for all $\alpha$
such that $(d^{\alpha})^{\vee, N} \geq d^{\alpha}$. If we repeat this
procedure, we see that there exist non-negative
rational numbers  $r_{\alpha}$ and $r'_{\alpha}$ 
for all $0 \leq \alpha \leq a$, that are uniquely determined, such that
\[
\beta(M) = \sum_{\alpha=0}^ar_{{\alpha}} \beta(d^{\alpha}) + 
\sum_{\alpha=0}^ar'_{{\alpha}} \beta((d^{\alpha})^{\vee,N}).
\]
Since $\beta(M)$ is $(s,N)$-self-dual, $r_{\alpha} = r'_{\alpha}$ for all
$\alpha$ and all the assertions follow immediately.
\end{proof}

\section{Main Theorem}
\label{sec:proof}

In this section, we prove Theorem~\ref{theorem:UpperBound}, which we
restate here for the sake of convenience.

\begin{theorem}
\upperboundtheorem
\end{theorem}

\begin{definition}
Let $d  = (0, d_1, \ldots, d_s)$ be a degree sequence such that $d_s \geq 
d_i + d_{s-i}$ for all $0 \leq i \leq s$.
Let $b_{d}  = \beta_0(d) + \beta_0(d^{\vee,d_s})$ and
\[
\Psi_{d} = \prod_{i=1}^k \min\left\{d_s-d_{s-i}, 
{\left\lfloor\frac{d_s}{2}\right\rfloor}\right\}
\prod_{i=k+1}^s \max\left\{d_{i}, 
{\left\lceil\frac{d_s}{2}\right\rceil}\right\}.
\]
\end{definition}

For two degree sequences $d = (d_0 <
\cdots <  d_s)$ and $d' = (d'_0 < \ldots < d'_s)$, we say that $d < d'$ if
$d_i \leq d'_i$ for all $0 \leq i \leq s$ and $d \neq d'$.

\begin{lemma}
\label{lemma:comparisonBettiPsi}
Let $d$ and $d'$ be degree sequences such that $d_0 = 0$ and
$d < d' \leq (d')^{\vee, d_s} < d^{\vee, d_s}$.
Then $\Psi_{d} > \Psi_{d'}$.
\end{lemma}

\begin{proof}
By induction on $\sum_{i} d'_i-d_i$,  we may assume,  
without loss of generality, that there exists $j$ such that
$d'_j = d_j+1$ and $d'_i = d_i$ for all $i \neq j$.
Moreover, if $1 \leq  i \leq s-k-1$, then $d_i$
does not figure in the expression for $\Psi_d$, so
we may assume that $j \geq s-k$. Additionally, $j \leq s-1$.
Let us rewrite $\Psi_d$ as
\begin{equation}
\label{equation:PsidReWritten}
\Psi_d = 
\prod_{i=s-k}^{s-1} \min\left\{d_s-d_i, 
{\left\lfloor\frac{d_s}{2}\right\rfloor}\right\}
\prod_{i=k+1}^s \max\left\{d_{i}, 
{\left\lceil\frac{d_s}{2}\right\rceil}\right\}.
\end{equation}
Two cases arise: $j < k+1$ and $j \geq k+1$. The first case is possible if
and only if $s=2k$ and $j=k$.
In this case, $d_k$ appears only once
in~\eqref{equation:PsidReWritten}, and since $d_k \leq
d_{s}-d_{s-k} = d_s-d_k$, we get 
$d_s - d_k \geq {\left\lfloor\frac{d_s}{2}\right\rfloor}$.
By the hypothesis that $d' \leq (d')^{\vee, d_s}$, 
$d_s - d_k - 1 \geq d_k+1$, so 
$d_s - d_k -1 \geq {\left\lfloor\frac{d_s}{2}\right\rfloor}$.
Hence
\[
\frac{\Psi_{d'}}{\Psi_d} = 
\frac{d_s - d_k - 1}{d_s - d_k} \leq 1.
\]
In the second case (i.e., $j \geq k+1$), $d_j$ appears twice in 
in~\eqref{equation:PsidReWritten}. However, note that 
\begin{align*}
\max\left\{d_j+1, {\left\lceil\frac{d_s}{2}\right\rceil}\right\} & \geq 
\max\left\{d_j, {\left\lceil\frac{d_s}{2}\right\rceil}\right\} \geq 
{\left\lceil\frac{d_s}{2}\right\rceil}, 
\qquad\text{and}, \\
\min\left\{d_s-d_j-1, {\left\lfloor\frac{d_s}{2}\right\rfloor}\right\} & \leq 
\min\left\{d_s-d_j, {\left\lfloor\frac{d_s}{2}\right\rfloor}\right\} \leq 
{\left\lfloor\frac{d_s}{2}\right\rfloor}
\end{align*}
so 
\[
\Psi_{d'} = \Psi_d 
\frac{\max\{d_j+1, {\left\lceil\frac{d_s}{2}\right\rceil}\}
\min\{d_s-d_j-1, {\left\lfloor\frac{d_s}{2}\right\rfloor}\}}
{\max\{d_j, {\left\lceil\frac{d_s}{2}\right\rceil}\}
\min\{d_s-d_j, {\left\lfloor\frac{d_s}{2}\right\rfloor}\}}
\leq \Psi_d.
\qedhere
\]
\end{proof}

The following proposition shows that Theorem~\ref{theorem:UpperBound} holds
for symmetrized pure Betti tables. 

\begin{proposition}
\label{proposition:upperBdHoldsForSymmPureBettiTables}
Let $d  = (0, d_1, \ldots, d_s)$ be a degree sequence such that 
$d \leq d^{\vee, d_s}$.
Then $b_{d} \Psi_{d}  \geq 2$.
\end{proposition}

\begin{proof}
We will prove this by induction on 
$\sum_{i} (d_s-d_{s-i}-d_i)$, which is non-negative by our hypothesis.
If $\sum_{i} (d_s-d_{s-i}-d_i) = 0$ (equivalently, $d = d^{\vee, d_s}$), 
then the assertion is true. If $d < d^{\vee, d_s}$, 
then there exists $j > k$ such that 
$d_j < d_{s}-d_{s-j}$. Pick $j$ to be maximal with this property. Let
$d' = (0, d_1, \cdots, d_{j-1}, d_j+1, d_{j+1}, \cdots, d_s)$.
Then $d' \leq (d')^{\vee}$ and $\sum_{i} (d_s-d_{s-i}-d_i) >
\sum_{i} (d'_s-d'_{s-i}-d'_i)$, so by induction,
$b_{d'} \Psi_{d'}  \geq \frac{2}{s!}$. 

We now show that $b_{d} \Psi_{d}  \geq b_{d'} \Psi_{d'}$. If 
$b_{d}  \geq b_{d'}$, then it is true by
Lemma~\ref{lemma:comparisonBettiPsi}. Hence suppose that 
$b_{d}  < b_{d'}$. In particular 
$d_j > \frac{d_s}{2}$, and hence $j > k$. 
Therefore
\[
\frac{\Psi_{d'}}{\Psi_d} = \frac{(d_s-d_j-1)(d_j+1)}{(d_s-d_j)d_j}.
\]
Therefore it suffices to show that 
$b_{d'}{(d_s-d_j-1)(d_j+1)} \leq b_d{(d_s-d_j)d_j}$. 
Let 
\[
\xi_1 = \prod_{\substack{i=0\\ i\neq j}}^{s-1} \frac{1}{d_s-d_{i}}
\quad\text{and}\quad
\xi_2 = \prod_{\substack{i=1\\ i\neq j}}^{s} \frac{1}{d_{i}}.
\]
Then $b_{d'}{(d_s-d_j-1)(d_j+1)} = (d_j+1)\xi_1 + (d_s-d_j-1)\xi_2$ and
$b_{d}{(d_s-d_j)d_j} = d_j\xi_1 + (d_s-d_j)\xi_2$. Then 
$b_d{(d_s-d_j)d_j} - b_{d'}{(d_s-d_j-1)(d_j+1)} = \xi_2 - \xi_1$. We can
see that $\xi_2 - \xi_1 \geq 0$ 
by noting that, for all $i$,  
the $i$th element of the sequence
$(d_s - d_{s-1}) < \cdots < (d_s - d_{j+1}) < (d_s - d_{j-1}) < \cdots <
(d_s-d_1) < d_s$ is at least as large as the 
$i$th element of $d_1 < \cdots < d_{j-1} < d_{j+1} < \cdots < d_s$, since
$j > k$.
\end{proof}

\begin{proof}[Proof of Theorem~\ref{theorem:UpperBound}]
Pick degree sequences $d^\alpha$ and non-negative rational numbers 
$r_\alpha$ as in Proposition~\ref{proposition:moduleBettiDec}. 
We need to show that $s!e(M) \leq \beta_0(M)\Psi_{d^{0}}$.
We get this as follows:
$s!e(M) = s! \sum_{\alpha} r_{\alpha} e(\beta_{sym}(d^{\alpha}))= \sum_{\alpha} 2 r_{\alpha}
\leq \sum_{\alpha} r_{\alpha} b_{d^{\alpha}} \Psi_{d^{\alpha}}
\leq 
\left(\sum_{\alpha} r_{\alpha} b_{d^{\alpha}}\right) \Psi_{d^{0}} 
= \beta_0(M)\Psi_{d^{0}}$,  where the two equalities follow from
Proposition~\ref{proposition:moduleBettiDec}
and~\eqref{equation:multPureBettiTable}, 
and the two inequalities follow from
Proposition~\ref{proposition:upperBdHoldsForSymmPureBettiTables}
and from Lemma~\ref{lemma:comparisonBettiPsi}, respectively.  Since $d^{0} = (0, t_1, t_2, \ldots, t_s) = (1, m_1, \ldots , m_s=d_s)$, we get the result. 
\end{proof}

\subsection*{Lower bounds} 
We have not been able to find an analogous generalization of the lower
bound in~\cite[Theorem~4]{S} (see~\eqref{equation:SrinivasanQuasiPure}) to
the non-quasi-pure case.
As an example, consider
$R= k[x,y,z]$ and $I = (yz, xz, xy, y^7-z^7, x^7-z^7)$. Note that
$e(R/I) = 20$.  The minimal free resolution of $R/I$ is:
\[
0 \longrightarrow R(-10) \longrightarrow R(-8)^3 \oplus
R(-3)^2\longrightarrow R(-7)^2 \oplus R(-2)^3  \longrightarrow R
\longrightarrow 0
\]
Indeed $e(R/I) \ngeq \frac{m_1M_2M_3}{6}=\frac{160}{6}$. 
It will be interesting to obtain strong lower bounds.

\subsection*{Special bounds for codimension 3}
 
Suppose that $R = \Bbbk[x,y,z]$ and $I$ is a homogeneous $(x,y,z)$-primary
$R$-ideal  such that $R/I$ is Gorenstein.
Write $N_1 = \max \{j : \beta_{1,j}(R/I) > \beta_{2,j}(R/I)\}$.
Migliore, Nagel and Zanello show that (see
\cite[Theorem~3.1]{MiNaZaCodimThreeGor08})
\begin{equation}
\label{equation:MNZbound}
e(R/I) \leq \frac{N_1T_2d_3}{6}.
\end{equation}
The bound in~\eqref{equation:MNZbound} is not comparable with
that from Theorem~\ref{theorem:UpperBound}.  We give examples to show this. To begin with, note that, in the
codimension-three situation, the bound from
Theorem~\ref{theorem:UpperBound} can be rewritten as
\begin{equation}
\label{equation:codimThreeOurBound}
e(R/I) \leq 
\begin{cases}
\frac{M_1m_2m_3}{6}, & 
\text{if $R/I$ has a quasi-pure resolution, i.e., $M_1 \leq m_2$},\\
\frac{{\left\lfloor\frac{m_3}{2}\right\rfloor}
{\left\lceil\frac{m_3}{2}\right\rceil}m_3}{6}, & \text{otherwise}.
\end{cases}
\end{equation}

Suppose that $R/I$ has a quasi-pure resolution. If the degree of the socle
of $R/I$ is even, or equivalently, $m_3$ is odd, then $M_1 = m_2-1$. 
Hence $\beta_{1,M_1} > 0 = \beta_{2,M_1}$, so $N_1 = M_1$. In this case,
the bound from~\eqref{equation:codimThreeOurBound} is strictly smaller than
the bound from~\eqref{equation:MNZbound}. On the other hand,
if the socle lives in an odd degree, then $M_1=m_2$ and $\beta_{1,M_1} =
\beta_{2,M_1}$, so $N_1 = \max\{j < M_1 : \beta_{1,j} \neq 0\}$. From the
next two examples,
we see that neither bound performs better than the other in this situation.
Consider, first, $I = (x^2, y^2, z^4)$ in $\Bbbk[x,y,z]$. 
Since $N_1 = 2$, the bound $\frac{2 \cdot 6 \cdot 8}{6} = 16$
of~\eqref{equation:MNZbound} is smaller
than the bound $\frac{4 \cdot 4 \cdot 8}{6} = 21 \frac13$
of~\eqref{equation:codimThreeOurBound}.
Now consider the 
$4 \times 4$ Pfaffians of a skew-symmetric map from 
$R(-6)\oplus R(-7)\oplus R(-7)\oplus R(-8)\oplus R(-8)$ to
$R(-4) \oplus R(-4) \oplus R(-5) \oplus R(-5) \oplus R(-6)$, constructed
using the following \texttt{Macaulay2}~\cite{M2} code.
\begin{minipage}{\linewidth}
\begin{verbatim}

        R = QQ[x,y,z];
        random(R^{-4, -4, -5, -5, -6}, R^{-8, -8, -7, -7, -6});
        phi = oo-transpose(oo);
        I = pfaffians(4,phi);
        betti res I
    
\end{verbatim}
\end{minipage}
For this example, $N_1 = 5$. Hence the bounds
in~\eqref{equation:codimThreeOurBound} and~\eqref{equation:MNZbound},
respectively, are $\frac{5 \cdot 8 \cdot 12}{6} = 80$ and
$\frac{6 \cdot 6 \cdot 12}{6} = 72$. 

The next two examples show that
similar behaviour can be expected in the case non-quasi-pure resolutions.
If $I = (yz, xz, x^3+x^2y-xy^2-2y^3, x^2y^2-y^4, xy^4-z^5)$, then the bound
of~\eqref{equation:MNZbound} is 
$\frac{2 \cdot 6 \cdot 8}{6} = 16$ and that
of~\eqref{equation:codimThreeOurBound} is
$\frac{4 \cdot 4 \cdot 8}{6} = 21\frac13$.
If $I = (yz, xz, x^2y^2-xy^3+y^4, x^4+x^3y+2xy^3, y^6-z^6)$, 
then the bounds are
$\frac{6\cdot 7 \cdot 9}{6} = 63$~\eqref{equation:MNZbound} and 
$\frac{4 \cdot 5 \cdot 9}{6} = 30$~\eqref{equation:codimThreeOurBound}.

\providecommand{\bysame}{\leavevmode\hbox to3em{\hrulefill}\thinspace}
\providecommand{\MRhref}[2]{%
  \href{http://www.ams.org/mathscinet-getitem?mr=#1}{#2}
}
\providecommand{\href}[2]{#2}

\end{document}